\documentclass{amsart}
\usepackage{amsmath,amsthm,amsfonts,amssymb,amscd,amsbsy,pgf,tikz,microtype}
\usepackage{caption,enumerate}
\usepackage{dsfont,lscape}

\usepackage[all]{xy}
\usepackage{hyperref}

\hypersetup{
    pdftoolbar=true,
    pdfmenubar=true,
    pdffitwindow=false,
    pdfstartview={FitH},
    pdftitle={},
    pdfauthor={},
    pdfsubject={},
    pdfkeywords={}
    pdfnewwindow=true,
    colorlinks=true, 
    linkcolor=blue,
    citecolor=blue,
    urlcolor=black,
}

\newcommand{\dd}{\mathrm{d}}

\newcommand{\triv}{{\rm{triv}}}

\newcommand{\Ric}{\operatorname{Ric}}

\newcommand{\Btriv}{\mathcal B_{\mathrm{triv}}}

\newcommand{\N}{\mathds N}
\newcommand{\Z}{\mathds Z}
\newcommand{\R}{\mathds R}
\newcommand{\C}{\mathds C}

\renewcommand{\S}{\mathsf{S}}
\newcommand{\Ss}{\mathds{S}}
\newcommand{\T}{\mathsf{T}}

\newcommand{\g}{\mathrm g}

\allowdisplaybreaks

\newtheorem{theorem}{Theorem}[]
\newtheorem{lemma}[theorem]{Lemma}
\newtheorem{proposition}[theorem]{Proposition}

\newtheorem*{prob*}{\sc Problem}

\newtheorem*{mainthm*}{\sc Theorem}
\newtheorem*{maincor*}{\sc Corollary}

\theoremstyle{definition}

\theoremstyle{remark}
\newtheorem{remark}[theorem]{Remark}

\title{Bifurcations of Clifford tori in ellipsoids}


\author[R. G. Bettiol]{Renato G. Bettiol}
\address[R. G. Bettiol]{\newline
\indent \!\!\!\begin{tabular}{lll}
CUNY Lehman College & & CUNY Graduate Center \\
Department of Mathematics & & Department of Mathematics \\
250 Bedford Park~Blvd W & & 365 Fifth Avenue \\
Bronx, NY, 10468, USA & & New York, NY, 10016, USA
\end{tabular}
}
\email{r.bettiol@lehman.cuny.edu}

\author[P. Piccione]{Paolo Piccione}
\address[P. Piccione]{\newline
\indent Great Bay University \newline
\indent Department of Mathematics \newline
\indent School of Sciences \newline
\indent Dongguan, Guangdong 523000, China}

\address{\emph{Permanent address:} \newline
\indent Universidade de S\~ao Paulo \newline
\indent Departamento de Matem\'atica \newline
\indent Rua do Mat\~ao, 1010 \newline
\indent S\~ao Paulo, SP, 05508-090, Brazil}
\email{piccione@ime.usp.br}

\allowdisplaybreaks
\numberwithin{equation}{section}
\numberwithin{theorem}{section}

\date{\today}

\begin{document}

\begin{abstract}
We prove that $3$-dimensional ellipsoids invariant under a $2$-torus action contain infinitely many distinct immersed minimal tori, with at most one exception. These minimal tori bifurcate from the $2$-torus orbit of largest volume at a dense set of eccentricities, and remain invariant under a circle.
\end{abstract}

\maketitle

\section{Introduction}

Despite recent spectacular advances in the existence theory of minimal surfaces, e.g., obtained via min-max theory and doubling constructions, many fundamental questions remain unanswered. In particular, while min-max theory applies broadly, yielding infinitely many embedded minimal surfaces on closed manifolds~\cite{song}, there is usually not much control on the topology of these minimal surfaces. For example, it is still unknown if the set of embedded minimal surfaces of a given genus $g\geq2$ in the round $\Ss^3$ is finite up to congruence, a question raised by Yau~\cite{yau-prob}. For genus $g=0$ and $g=1$, there is only one such minimal surface up to congruence, as shown by Almgren~\cite{almgren} and Brendle~\cite{brendle}, respectively. Recently, Ketover~\cite{ketover} proved that the number of such minimal surfaces diverges as $g\nearrow+\infty$. Another intriguing open question is whether every metric on $\Ss^3$ admits at least 5 embedded minimal tori, as conjectured by White~\cite{white-tori}, who verified this for almost round metrics. In the same paper, White showed that metrics on $\Ss^3$ with $\Ric\succ0$ carry at least one such torus. 
Significant progress towards proving White's conjecture has been recently announced in \cite{chu-li,li-wang}.

A natural approach to produce minimal surfaces with prescribed topology is to use bifurcation theory. Namely, given a $1$-parameter family $\Sigma_a$ of minimal surfaces in $(M^3,\g_a)$, a jump in the Morse index of $\Sigma_a$ as $a$ crosses a degeneracy instant $a_*$ generally corresponds to the existence of a bifurcation branch of other (isotopic, but noncongruent) minimal surfaces in $(M^3,\g_a)$ for $a$ near $a_*$. Under some conditions, this branch is not only local, but global, i.e., continues to exist for $a$ far from $a_*$.

The typical setup to implement this approach is given by deformations $(M^3,\g_a)$ of manifolds with many symmetries, which retain a family of minimal surfaces $\Sigma_a$. For instance, ellipsoidal deformations $(\Ss^3,\g_a)\subset \R^4$ of the round sphere retain 4 ``planar'' minimal spheres given by the intersection with a coordinate hyperplane in~$\R^4$. If the minimal spheres $\Sigma_a$ in one of these families are rotationally invariant, then they bifurcate into arbitrarily many \emph{nonplanar} pairwise noncongruent minimal spheres as the ellipsoid $(\Ss^3,\g_a)$ becomes sufficiently elongated~\cite{ellipsoids}. Similarly, certain ellipsoidal deformations retain a family of minimal tori. Namely, consider 
\begin{equation*}
\Ss^3(a,b):=\left\{(z,w)\in\C^2: \frac{|z|^2}{a^2}+\frac{|w|^2}{b^2}=1\right\},
\end{equation*}
with the isometric action of the torus $\T^2=\S^1 \times \S^1$, where each factor $\S^1\times \{1\}$ and $ \{1\}\times \S^1$ acts via multiplication by unit complex numbers, on $z$ and $w$, respectively. In view of the case $a=b$ of the round sphere, we shall refer to the ``middle'' $\T^2$-orbit
\begin{equation*}
    \Sigma(a,b):=\left\{(z,w)\in \Ss^3(a,b) :\frac{|z|^2}{a^2}=\frac{|w|^2}{b^2}=\frac12 \right\},
\end{equation*}
as the \emph{Clifford torus} in $\Ss^3(a,b)$. 
For all $a,b>0$, the torus $\Sigma(a,b)$ has maximal volume among principal $\T^2$-orbits in $\Ss^3(a,b)$, and is hence a minimal surface. 

In this paper, we investigate minimal tori in $\Ss^3(a,b)$ that bifurcate from $\Sigma(a,b)$, or from one of its  finite coverings, as the ratio $a/b$ varies. Without loss of generality, as minimal surfaces in $\Ss^3(a,b)$ can be identified with those in $\Ss^3(\lambda a,\lambda b)$ for all $\lambda>0$, we henceforth fix $b=1$, and use $a>0$ as the ``eccentricity'' parameter. To simplify notation, we write $\Ss^3_a:=\Ss^3(a,1)$ and $\Sigma_a:=\Sigma(a,1)$. Our main result~is:

\begin{mainthm*}
For $a_*\in \mathfrak b=\{q/\sqrt{4-q^2} : q\in\mathds Q\cap (0,2) \}$, which is dense in $(0,+\infty)$, 
a bifurcation branch of $\S^1$-invariant immersed minimal tori in $\Ss^3_a$ stems from the Clifford torus $\Sigma_{a_*}$. Each such branch persists for all $a$ up to $0$ or $+\infty$, and different branches contain pairwise noncongruent tori in $\Ss^3_a$. The bifurcation branch that issues at $a_*=\frac{1}{\sqrt3}$ contains only embedded minimal tori, and persists until $a=0$.
\end{mainthm*}

Many geometric properties of these minimal tori in $\Ss^3_a$ can be inferred from the instant at which they bifurcate from $\Sigma_a$. Given $q\in\mathds Q\cap (0,2)$, let $0<j<2k$ be the unique integers such that $\gcd(j,k)=1$ and $q=j/k$, and let $\mathcal B_{(j,k)}$ be the branch of minimal tori that bifurcate from $\Sigma_a$ at $a_* = q/\sqrt{4-q^2}$. (More precisely, the tori in $\mathcal B_{(j,k)}$ bifurcate from a $k$-fold covering of $\Sigma_a$.) 
Each minimal immersion $\T^2\hookrightarrow \Ss^3_a$ in $\mathcal B_{(j,k)}$ maps the parallel circles $\{e^{i\theta}\}\times \S^1\subset\T^2$ into $\{1\}\times \S^1$-orbits in~$\Ss^3_a$. Tori in $\mathcal B_{(j,k)}$, other than $\Sigma_{a_*}$, intersect $\Sigma_a$ along the image of $2j$ such circles, and have linking number $k$ with the closed geodesic $\{(0,w) \in \Ss^3_a : |w|=1\}$. This implies that the branches $\mathcal B_{(j,k)}$ are pairwise disjoint, see Proposition~\ref{prop:disjoint} for details. All tori in $\mathcal B_{(1,1)}$ are embedded, but none of the tori in $\mathcal B_{(j,k)}$ with $(j,k)\neq (1,1)$ are embedded. The latter are (Alexandrov) immersed and self-intersect along the image of $k-1$ circles contained in the $2$-sphere $\Ss^3_a \cap\{\operatorname{Im}z=0\}$, which are $\{1\}\times\S^1$-orbits.

Owing to our $\S^1$-invariant setup, all branches $\mathcal B_{(j,k)}$ can be realized as closed connected subsets of the strip $\{(a,s)\in (0,+\infty)\times (-1,1)\}$, where $s\in (-1,1)$ parametrizes $\{1\}\times\S^1$-invariant minimal immersions $\R\times \S^1 \hookrightarrow \Ss^3_a$ that coincide with the Clifford torus $\Sigma_a$ if $s=0$, close up as tori if $(a,s)\in \mathcal B_{(j,k)}$ for some $j,k$, and degenerate to the totally geodesic $2$-sphere $\Ss^3_a\cap \{\operatorname{Re} z =0\}$ as $s\searrow -1$, and to the closed geodesic $\{(z,0) \in \Ss^3_a : |z|=a\}$ as $s\nearrow1$. Using this, we prove that $\mathcal B_{(j,k)}$ are noncompact and contain points $(a,s)$ with $a$ arbitrarily close to $0$ or $+\infty$. 

The unit round sphere $\Ss^3_1$ contains infinitely many pairwise noncongruent immersed minimal tori which are $\S^1$-invariant, first described by \^Otsuki~\cite{otsuki}, see also \cite[Thm.~8]{hsiang-lawson} and \cite[Thm.~1.4]{brendle-survey}. This fact can be generalized to the ellipsoids $\Ss^3_a$ as a consequence of the above Theorem:

\begin{maincor*}
The ellipsoid $\Ss^3_a$ contains infinitely many pairwise noncongruent $\S^1$-invariant immersed minimal tori for all $a \in (0,+\infty)$ except possibly for one value in $\big(\frac{1}{\sqrt3},+\infty\big)$. 
If $a\in \big(0,\frac{1}{\sqrt3}\big)$, then at least one of these minimal tori is embedded and not congruent to the Clifford torus $\Sigma_a$.
\end{maincor*}

The value $a=\frac{1}{\sqrt3}$ has a special role because tori in $\mathcal B_{(1,1)}$ are embedded. Namely, $\mathcal B_{(1,1)}$ cannot enter the region  $a>1$, otherwise it would yield an embedded minimal torus in the round sphere $\Ss^3_1$ not congruent to the Clifford torus, in contradiction with Brendle's proof~\cite{brendle} of the Lawson conjecture. In other words, $a=1$ is a barrier for $\mathcal B_{(1,1)}$, and $\mathcal B_{(1,1)}$ is then itself a barrier for all $\mathcal B_{(j,k)}$ that bifurcate at $a_*<\frac{1}{\sqrt3}$. Thus, these are nested in the strip $\{(a,s)\in (0,+\infty)\times (-1,1)\}$, and so contain points with $a$ arbitrarily close to $0$. 
Either this is the case for all branches (which is natural to expect), or else a branch bifurcating at some $a_*>\frac{1}{\sqrt3}$, and hence all subsequent branches, do not contain points with $a$ close to $0$, but instead have points with arbitrarily large $a$. If the latter occurs, in principle, there could be (at most) one value of $a\in \big(\frac{1}{\sqrt3},+\infty\big)$ for which $\{a\}\times (-1,1)$ does not intersect infinitely many branches $\mathcal B_{(j,k)}$. Because of this possibility, which we cannot currently rule out, the above Corollary states \emph{except possibly for one value in} $\big(\frac{1}{\sqrt3},+\infty\big)$.

Our analysis also yields a local uniqueness counterpart to the above Theorem. Given $k\in \N$, let $\mathfrak b_{k}:=\{(j/k)/\sqrt{4-(j/k)^2} : 0<j<2k,\, \gcd(j,k)=1 \}$, and note that $\mathfrak b=\bigcup_{k\in\N}\mathfrak b_{k}$. If $a\notin \mathfrak b_{k}$, then the $k$-fold covering of the Clifford torus $\Sigma_a$ is locally unique among $\{1\}\times\S^1$-invariant immersed minimal tori in $\Ss^3_a$ up to congruence. In particular, since $\mathfrak b_{1}=\{\frac{1}{\sqrt3}\}$, we conclude that the Clifford torus $\Sigma_a$, $a\neq \frac{1}{\sqrt3}$, is locally unique among $\{1\}\times\S^1$-invariant embedded minimal tori in $\Ss^3_a$ up to congruence. For details, see Proposition~\ref{prop:loc-bif} (ii).

We remark that the round sphere $\Ss^3_1$ also contains many immersed minimal tori that are \emph{not} $\S^1$-invariant. In fact, Hitchin~\cite{hitchin} discovered an explicit correspondence between immersed minimal tori in $\Ss^3_1$ and ``spectral data'', which are certain hyperelliptic curves endowed with a line bundle and two meromorphic differentials with integer periods. Using this correspondence and integrable systems techniques, relatively explicit examples of such tori in $\Ss^3_1$ were produced by Carberry~\cite{carberry}.

Beyond ellipsoidal deformations, there are other natural families of deformations of the round $\Ss^3$ that retain several minimal surfaces, and are hence well-suited for the bifurcation approach put forth in this paper. For instance, Berger spheres $\Ss^3(\tau)$, $\tau>0$, are homogeneous deformations of the round sphere $\Ss^3(1)$ obtained scaling the Hopf circles $\S^1\to \Ss^3\to \C P^1$ by $\tau>0$. Their isometry group $\mathsf U(2)$ contains a torus $\T^2$ that acts by cohomogeneity one, and the $\T^2$-orbit of maximal volume in $\Ss^3(\tau)$ is an embedded minimal torus, which coincides with the Clifford torus for $\tau=1$. These minimal tori bifurcate infinitely many times as $\tau\nearrow+\infty$, as shown in \cite{LLP}, though in a local ``cluster-point'' sense that is weaker than the one used here. Embedded minimal tori in $\Ss^3(\tau)$ not congruent to the Clifford torus have also been found in~\cite{torralbo}. 
On the other hand, some families of minimal surfaces never bifurcate; e.g., as $\Ss^3(\tau)$ are homogeneous, they each admit a \emph{unique} (up to congruence) immersed minimal $2$-sphere~\cite{mmpr2}.

This paper is organized as follows. In Section~\ref{sec:setup}, we use the symmetry reduction method of Hsiang--Lawson~\cite{hsiang-lawson} to identify $\{1\}\times\S^1$-invariant minimal tori in $\Ss^3_a$ and closed geodesics on a certain Riemannian $2$-disk $\Omega_a$ with singular boundary. The second variation of (iterates of) the closed geodesic $\gamma_{a,0}$ in $\Omega_a$ that corresponds to the Clifford torus $\Sigma_a$ is analyzed in Section~\ref{sec:bif}. We restrict to perturbations of $\gamma_{a,0}$ that are invariant under a reflection in $\Omega_a$ to avoid multiplicities issues, allowing us to use the simple eigenvalue bifurcation theorem of Crandall--Rabinowitz~\cite{crandall-rabinowitz} for the local result (Proposition~\ref{prop:loc-bif}). Combining this local result with the discrete-valued invariant in Proposition~\ref{prop:disjoint} and the global bifurcation theorem of Rabinowitz~\cite{rabinowitz}, we prove the Theorem and Corollary stated above.

\subsection{Acknowledgements}
The first-named author was supported by the National Science Foundation CAREER grant DMS-2142575. The second-named author was supported by grants from CNPq and Fapesp (2022/16097-2, 2022/14254-3), Brazil.

\section{Geometric setup}
\label{sec:setup}

\subsection{Symmetry reduction}
Consider the isometric $\S^1$-action on $\Ss^3_a\subset \C^2$ given by complex multiplication in the second coordinate, i.e., the action of $\{1\}\times\S^1\subset\T^2$. The orbit through $(z,w)\in \Ss^3_a$ is a circle of radius $r=|w|=\sqrt{1-\frac{|z|^2}{a^2}}$ if $|w|>0$, and a fixed point if $|w|=0$. Moreover, the orbit space is isometric to
\begin{equation}\label{eq:quotient}
\Ss^3_a/(\S^1) = \left\{ (z,r)\in\C\times\R : \frac{|z|^2}{a^2}+ r^2 =1, \, r\geq 0 \right\},
\end{equation}
with boundary $\partial(\Ss^3_a/\S^1)=\{(z,0)\in \Ss^3_a/ \S^1\}$ given by the circle of $\S^1$-fixed points in $\Ss^3_a$, and the volume function of $\S^1$-orbits is given by
\begin{equation}\label{eq:vol-fct}
V\colon \Ss^3_a/ \S^1\longrightarrow\R, \quad V(z,r)=2\pi r=2\pi \sqrt{1- |z|^2/a^2}.
\end{equation}

Consider the Riemannian $2$-disk endowed with the conformal metric
\begin{equation*}
\Omega_{a}:=\big(\Ss^3_a/ \S^1, V^2\, \check{\g}_{a}\big),
\end{equation*}
where $\check{\g}_{a}$ is the Riemannian metric of \eqref{eq:quotient}.
The symmetry reduction procedure of Hsiang--Lawson~\cite{hsiang-lawson}, cf.~\cite[Prop.~3.1]{ellipsoids}, readily implies the following:

\begin{proposition}\label{prop:reduction}
An $\{1\}\times\S^1$-invariant surface $\Sigma$ in $\Ss^3_a$ is minimal if and only if its quotient $\Sigma/\S^1$ is a geodesic in $\Omega_{a}$. In particular, $\Sigma$ is an $\{1\}\times\S^1$-invariant minimal torus in $\Ss^3_a$ if and only if $\Sigma/\S^1$ is a closed geodesic in the interior of $\Omega_{a}$, and $\Sigma$ is embedded if and only if $\Sigma/\S^1$ is embedded (i.e., a simple closed geodesic).
\end{proposition}

\subsection{\texorpdfstring{Geodesics in $\Omega_{a}$}{Geodesics in the renormalized orbit space}}
Since the isometric actions on $\Ss^3_a$ of the circle subgroups $\S^1\times \{1\}$ and $\{1\}\times\S^1$ of $\T^2$ commute, and \eqref{eq:vol-fct} is $\S^1\times \{1\}$-invariant, it follows that the metric $V^2\, \check{\g}_{a}$ of $\Omega_{a}$ is invariant with respect to the induced $\S^1\times \{1\}$-action. 
In order to write this metric in polar coordinates $(\rho,\theta)$ as $\dd \rho^2+\varphi(\rho)^2\dd\theta^2$, we  parametrize the orbit space \eqref{eq:quotient} with $X(\phi,\theta)=\big(x(\phi,\theta),y(\phi,\theta),r(\phi,\theta)\big)$, where
\begin{equation*}
    x(\phi,\theta)=a\cos\theta\sin\phi, \quad y(\phi,\theta)=a\sin\theta\sin\phi, \quad r(\phi,\theta)=\cos\phi,
\end{equation*}
and $\phi\in \big[0,\frac\pi2\big], \, \theta\in[0,2\pi]$. Thus, from \eqref{eq:vol-fct}, we have that the metric of $\Omega_{a}$ is
\begin{equation}\label{eq:conf-metric-polarcoord}
\begin{aligned}
  V^2\, \check{\g}_{a}&=4\pi^2\cos^2\phi\,(\dd x^2+\dd y^2 +\dd r^2)   \\
  &=4\pi^2\cos^2\phi\,\big((a^2\cos^2\phi+\sin^2\phi)\,\dd\phi^2 + a^2\sin^2\phi\,\dd\theta^2\big)\\
  &=\dd\rho^2+ \varphi(\rho)^2\dd\theta^2,
\end{aligned}
\end{equation}
where 
\begin{equation}\label{eq:drho-varphi}
\dd\rho=2\pi \cos(\phi)\sqrt{a^2\cos^2\phi+\sin^2\phi}\,\dd\phi, \;\; \text{and} \;\; \varphi(\rho)=\pi a \sin(2\phi(\rho)).    
\end{equation}
Here, $\phi(\rho)$ is the inverse of the arclength function $\rho\colon\left[0,\frac\pi2\right]\to [0,L_a]$, given by $\rho(\phi)=2\pi \int_0^\phi \cos(\xi)\sqrt{a^2\cos^2\xi+\sin^2\xi}\,\dd\xi$, with $L_a=\rho(\frac\pi2)$. 
Clearly, $\varphi(\rho)>0$ for all $\rho \in (0,L_a)$, and $\varphi(0)=\varphi(L_a)=0$. Since $\rho'(0)=2\pi a$ and $\lim\limits_{\phi\nearrow \frac\pi2}\rho'(\phi)=0$, we have that $\varphi'(0)=1$ and $\lim\limits_{\rho\nearrow L_a}\varphi'(\rho)=-\infty$, corresponding to the fact that \eqref{eq:conf-metric-polarcoord} is smooth at the central point $O=\{\rho=0\}$ but singular at $\partial \Omega_{a}=\{\rho=L_a\}$.

\subsubsection{Geodesic equation}
A routine computation with \eqref{eq:conf-metric-polarcoord} shows that the curve $\gamma(t)=(\rho(t),\theta(t))$ is a geodesic in $\Omega_{a}$ if and only if it satisfies the system of ODEs
\begin{equation}\label{eq:geod-eqn}
\ddot\rho -\varphi(\rho) \varphi'(\rho)\,\dot\theta^2=0 \quad \text{and}\quad 
    \ddot\theta+2 \frac{ \varphi'(\rho)}{\varphi(\rho)}\,\dot\rho\,\dot\theta =0.
\end{equation}
Integrating the second equation above yields a conserved quantity, whose constant value depends on the geodesic $\gamma$, given by
\begin{equation}\label{eq:conserved}
\dot\theta\,\varphi(\rho)^2=c(\gamma),
\end{equation}
which corresponds to the fact that $\frac{\partial}{\partial \theta}$ is a Killing field in $\Omega_{a}$.
Thus, along a geodesic, $\theta$ is either \emph{constant} (radial geodesic) or \emph{monotonic}. This allows us to reparametrize nonradial geodesics as $\gamma(\theta)=(\rho(\theta),\theta)$, though we shall not make use of this fact.

For each $\theta_*\in [0,2\pi)$, we define the \emph{radial geodesic segment} issuing from $O$,
\begin{equation}\label{eq:radialgeod}
\sigma(\theta_*):=\{ (\rho,\theta_*)\in\Omega_{a} : 0<\rho <L_a\},
\end{equation}
and, for each $\theta_*\in [0,\pi)$, the \emph{diameter}
\begin{equation}\label{eq:diam}
D(\theta_*):=\overline{\sigma(\theta_*)\cup\sigma(\theta_*+\pi)}.
\end{equation}
Note that any geodesic in $\Omega_a$ going through $O$ must be a diameter; by \cite[Thm.~3.6]{ellipsoids} these are the only geodesics to reach $\partial\Omega_a$. The reflection $\tau_{\theta_*}\colon\Omega_a\to\Omega_a$ about $D(\theta_*)$ is an isometry, and the preimage of $D(\theta_*)$ under the quotient map $\Ss^3_a\to\Omega_a$ is a totally geodesic (planar) $2$-sphere in $\Ss^3_a$. From \eqref{eq:conserved}, every nonradial geodesic crosses (transversely) each $\sigma(\theta_*)$, see Lemma~\ref{lemma:transv}. 

\subsubsection{Clifford geodesic}\label{subsec:cliffordgeod}
According to \eqref{eq:geod-eqn}, a curve $\gamma(t)=(\rho(t),\theta(t))$ such that $\rho(t)\equiv\rho_{a,0}$ is constant, i.e., a circle of latitude $\rho_{a,0}$, is a geodesic in $\Omega_a$ if and only if $\varphi'(\rho_{a,0})=0$. As a consequence of \eqref{eq:drho-varphi}, there is only one such $\rho_{a,0}\in (0,L_a)$, namely the unique solution to $\phi(\rho_{a,0})=\frac\pi4$, since $\phi(\rho)$ is monotonic. The corresponding simple closed geodesic $\gamma_{a,0}(t):=(\rho_{a,0},t)$ is the image of the Clifford torus $\Sigma_a\subset \Ss^3_a$ under the quotient map $\Ss^3_a\to\Omega_{a}$, so we call it the \emph{Clifford geodesic}. Note that $\gamma_{a,0}$ has length $|\Sigma_a|=2\pi^2a$ and, from \eqref{eq:drho-varphi},
\begin{equation}\label{eq:valuesvarphi}
    \varphi(\rho_{a,0})=\pi a, \quad \text{ and }\quad \varphi''(\rho_{a,0})=-\tfrac{4 a}{\pi (a^2+1)}.
\end{equation}

\subsubsection{\texorpdfstring{Geodesics $\gamma_{a,s}$}{Orthogonal geodesics}}
Fix a real analytic function $\beta_a\colon (-1,1)\to (0,L_a)$ such that $\beta_a(0)=\rho_{a,0}$ for all $a$ and $\beta_a'(s)>0$. 
In particular, $(-1,1)\ni s\mapsto (\beta_a(s),0)\in\Omega_a$ is a parametrization of $\sigma(0)$ starting at $O$ and ending at $\partial\Omega_a$.
Let
\begin{equation}\label{eq:gamma_as}
\gamma_{a,s}(t)=\big(\rho_{a,s}(t),\theta_{a,s}(t)\big)    
\end{equation}
be the (maximal) geodesic in $\Omega_{a}$ starting orthogonally to $\sigma(0)$ at $\gamma_{a,s}(0)=(\beta_a(s),0)$ with $\dot\gamma_{a,s}(0)=(0,1)$. Since the reflection $\tau_0$ about $D(0)$ maps $\dot\gamma_{a,s}(0)$ to $-\dot\gamma_{a,s}(0)$, it follows that $\tau_0(\gamma_{a,s}(t))=\gamma_{a,s}(-t)$ for all $t\in\R$.
Note that, for $s=0$, we have that $\rho_{a,0}(t)\equiv\rho_{a,0}$ is constant and $\theta_{a,0}(t)=t$, since $\gamma_{a,0}$ is the Clifford geodesic.
For an illustration, see Figure~\ref{fig:omegaa}.

\begin{figure}[!ht]
\begin{tikzpicture}[scale=2]
    \filldraw[fill=gray!20!white,line width=0.3mm,densely dashed] (0,0) circle (1cm);
    \draw[line width=0.3mm,color=red] (0,0) circle (0.5cm);
    \draw[-latex] (0,-1.2) -- (0,1.4);
    \draw[-latex] (-1.4,0) -- (1.4,0);
    \draw[blue, line width=0.3mm] (0,0) -- (1,0);
    \draw[black!40!green, line width=0.3mm] (0,0) -- (-1,0);
    \draw (0.75,-.15) node {${\color{blue}\sigma(0)}$};
    \draw (-0.7,-.15) node {${\color{black!40!green}\sigma(\pi)}$};
    \draw (-.65,.35) node {${\color{orange}\gamma_{a,s}}$};
    \draw (.6,.3) node {${\color{red}\gamma_{a,0}}$};
    \draw (1.3,0) node [below] {$x$};
    \draw (0,1.3) node [right] {$y$};
    \fill[color=black] (0,0) circle (0.03);
    \draw (-0.1,-0.1) node {${O}$};
    \draw[domain=0:7/9*pi,samples=500,color=orange,line width=0.3mm] plot ({deg(\x)}:{.55-0.15*cos(\x r)});
\end{tikzpicture}
\caption{Schematic depiction of $\Omega_a$ with Clifford geodesic $\gamma_{a,0}$ in red, and portion of a geodesic $\gamma_{a,s}$, $s<0$, in orange.}\label{fig:omegaa}
\end{figure}

\begin{remark}\label{rem:allinbranches}
Every closed geodesic $\gamma(t)=(\rho(t),\theta(t))$ in $\Omega_a$ intersects some radial geodesic $\sigma(\theta_*)$ orthogonally, e.g., at the point where $\rho(t)$ is maximal, and hence is congruent to $\gamma_{a,s}$ via clockwise $\S^1\times \{1\}$-rotation by $\theta_*$.
\end{remark}

Using the variation $s\mapsto \gamma_{a,s}$ of $\gamma_{a,0}$ by geodesics, we linearize \eqref{eq:geod-eqn} and \eqref{eq:conserved}, obtaining the Jacobi equation for $(R_a,T_a)=\frac{\dd}{\dd s}(\rho_{a,s},\theta_{a,s})|_{s=0}$, which, in view of \eqref{eq:valuesvarphi}, is the following system of ODEs with constant coefficients:
\begin{equation}\label{eq:jacobi-eqn-quotient}
    \ddot R_a+\tfrac{4a^2}{a^2+1}\,R_a=0  \quad \text{and}\quad  (\pi a)^2 \,\dot T_a=\tfrac{\dd}{\dd s}c(\gamma_{a,s})\big|_{s=0}.
\end{equation}

\section{Bifurcations of the Clifford geodesic}
\label{sec:bif}

In this section, we classify the closed geodesics in $\Omega_a$ that bifurcate from the Clifford geodesic $\gamma_{a,0}$ as the parameter $a$ varies. By Proposition~\ref{prop:reduction}, this corresponds to classifying the $\{1\}\times\S^1$-invariant minimal tori in $\Ss^3_a$ that bifurcate from $\Sigma_a$. 

\begin{lemma}\label{lemma:transv}
Each $\gamma_{a,s}$ intersects (transversely) all radial geodesics $\sigma(k\pi)$, $k\in\N$. 
\end{lemma}

\begin{proof}
If $\gamma_{a,s}$ did not intersect $\sigma(k\pi)$, then $\lim_{t\to +\infty} \theta_{a,s}(t)=\theta_{\mathrm{max}}\in (0,k\pi]$ and $\lim_{t\to +\infty } \rho_{a,s}(t)=L_a$, i.e., $\gamma_{a,s}(t)$ converges to a point in $\partial \Omega_{a}$ as $t\nearrow + \infty$, hence $\varphi(\rho_{a,s}(t))\searrow 0$. This implies $\dot\theta \nearrow +\infty$ by \eqref{eq:conserved}, contradicting $\theta_{\mathrm{max}}\in (0,k\pi]$.
\end{proof}

By the Lemma~\ref{lemma:transv} and the Implicit Function Theorem, we have:

\begin{proposition}
For all $k\in\N$, there exists a real analytic positive function $\ell_k(a,s)$ such that $\theta_{a,s}(\ell_k(a,s))=k\pi$. In particular, $\gamma_{a,s}(\ell_k(a,s))\in D(0)$. 
\end{proposition}

\begin{remark}\label{rem:taua0}
We have that $\ell_k(a,0)=k\pi$ for all $a>0$, since $\theta_{a,0}(t)=t$. This also follows from
$\ell_k(a,0)=\frac{k|\Sigma_a|}{2\|\dot\gamma_{a,0}\|}$ as $|\Sigma_a|=2\pi^2a$ and $\|\dot\gamma_{a,0}\|=\pi a$ by \eqref{eq:valuesvarphi}.
\end{remark}

For each $k\in\N$, consider the real analytic function
\begin{equation}\label{eq:f}
    f_k\colon (0,+\infty) \times (-1,1) \longrightarrow \R, \quad f_k(a,s):= \dot\rho_{a,s}(\ell_k(a,s)),
\end{equation}
and note that $f_k(a,0)=0$ for all $a>0$, see Section~\ref{subsec:cliffordgeod}.

\begin{proposition}\label{prop:fk-zeros}
The following hold:
\begin{enumerate}[\rm (i)]
\item The geodesic $\gamma_{a,s}$ is closed if and only if $f_k(a,s)=0$ for some $k\in\N$;
\item The geodesic $\gamma_{a,s}$ is a simple closed geodesic if and only if $f_1(a,s)=0$;
\item The geodesic segment $\gamma_{a,s}([-\ell_k(a,s),\ell_k(a,s)])$ is a primitive closed geodesic (i.e., traces its image exactly once) with winding number $k$ around $O$ if and only if $f_k(a,s)=0$ and $f_{k'}(a,s)\neq 0$ for all $0<k'<k$.
\end{enumerate}
If $f_k(a,s)=0$, the preimage of $\gamma_{a,s}$ under the quotient map $\Ss^3_a\to\Omega_a$ is an $\{1\}\times\S^1$-invariant immersed minimal torus in $\Ss^3_a$, which is embedded if and only if $k=1$, and is congruent to (a covering of) the Clifford torus $\Sigma_a$ if and only if $s=0$.
\end{proposition}

\begin{proof}
Recall that $\tau_0(\gamma_{a,s}(t))=\gamma_{a,s}(-t)$ for all $t$, so $\gamma_{a,s}(\R)$ is invariant under $\tau_0$.
    If $f_k(a,s)=0$, then $\gamma_{a,s}([0,\ell_k(a,s)])$ is a geodesic segment that starts orthogonally from $\sigma(0)$ and ends orthogonally at $\sigma(k\pi)$, so $\gamma_{a,s}$ is a closed geodesic. 
    Conversely, if $\gamma_{a,s}$ is a closed geodesic, then any $\ell>0$ such that $\gamma_{a,s}(-\ell)=\gamma_{a,s}(\ell)$ and $\dot\gamma_{a,s}(-\ell)=\dot\gamma_{a,s}(\ell)$ must satisfy $\theta_{a,s}(\ell)=k\pi$ for some $k\in\N$, i.e., $\ell=\ell_k(a,s)$, because $\gamma_{a,s}(\ell)$ is a fixed point of $\tau_0$ hence lies on $D(0)$. Since $\gamma_{a,s}$ is smooth at that point, $\dot\gamma_{a,s}(\ell)$ must be orthogonal to $D(0)$, i.e., $f_k(a,s)=0$. This concludes the proof of (i), and shows that if $f_k(a,s)=0$ then $f_{mk}(a,s)=0$ for all $m\in\N$. By the same arguments, $\ell=\ell_k(a,s)$ is the \emph{smallest} $\ell>0$ as above if and only if $\gamma_{a,s}([-\ell,\ell])$ is a \emph{primitive} closed geodesic with winding number $k$ around $O$, which proves (ii) and (iii).

    The claims about the preimage of $\gamma_{a,s}$ in $\Ss^3_a$ follow from Proposition~\ref{prop:reduction}, and from the correspondence between congruence classes of $\{1\}\times\S^1$-invariant surfaces in $\Ss^3_a$ and congruence classes of their images in $\Omega_a$, i.e., orbits under $\mathsf S^1\times\{1\}$-rotations and reflections $\tau_{\theta_*}$, all of which map $\gamma_{a,s}$ to $\gamma_{a,0}$ if only if $s=0$.
\end{proof}

\begin{remark}
    If $f_1(a,s)=0$, then $\gamma_{a,s}$ is not only a \emph{simple} closed geodesic but also meets each radial geodesic $\sigma(\theta_*)$ exactly \emph{once}, i.e., bounds a star-shaped domain.
\end{remark}

\begin{remark}\label{rem:alex-imm}
    The minimal tori in Proposition~\ref{prop:fk-zeros} are \emph{Alexandrov immersed}, i.e., the immersion $\T^2\hookrightarrow \Ss^3_a$ is the restriction of an immersion $\S^1\times \overline{B^2}\hookrightarrow \Ss^3_a$ to $\partial (\S^1\times \overline{B^2})=\T^2$. Namely, $\T^2\hookrightarrow \Ss^3_a$ can be extended to an immersion $\S^1\times \overline{B^2}\hookrightarrow \Ss^3_a$ by filling in the corresponding $\{1\}\times\S^1$-orbits in $\Ss^3_a$ with $\{1\}\times\S^1$-invariant $2$-disks.
\end{remark}

In what follows, we study the zero sets $f_k^{-1}(0)\subset (0,+\infty) \times (-1,1)$, which contain the \emph{trivial branch} $\Btriv=\{(a,0) : a>0\}$ corresponding to the Clifford tori $\Sigma_a$. If $(a_*,0)$ is in the closure of $f_k^{-1}(0)\setminus\Btriv$, we say $a=a_*$ is a \emph{bifurcation instant} for $f_k$, and the connected component of the closure of $f_k^{-1}(0)\setminus\Btriv$ containing $(a_*,0)$ is called the \emph{bifurcation branch} for $f_k$ issuing from $(a_*,0)$. In particular, $\frac{\partial f_k}{\partial s}(a_*,0)=0$ by the Implicit Function Theorem, for otherwise points in $\Btriv$ would be locally unique as solutions to $f_k(a,s)=0$ near $a=a_*$. However, in general, this necessary condition is \emph{not} sufficient for bifurcation to occur at $a=a_*$.

\subsection{Local bifurcation}
A sufficient condition for $a=a_*$ to be a bifurcation instant is given by a classical theorem of Crandall and Rabinowitz~\cite{crandall-rabinowitz}, see \cite[Thm~2.2]{ellipsoids} for a formulation adapted to the present setup, which yields:

\begin{proposition}\label{prop:loc-bif}
    Let $k\in\N$ be fixed. For each integer $0<j<2k$, let 
\begin{equation}\label{eq:ajk}
    a^j_k:=\frac{(j/k)}{\sqrt{4-(j/k)^2}}.
\end{equation}
    \begin{enumerate}[\rm (i)]
     \item There exists an open neighborhood $U^j_k$ of $\big(a^j_k,0\big)$ in $(0,+\infty)\times (-1,1)$ and a real analytic curve $(-\varepsilon,\varepsilon)\ni t\mapsto \big(a^j_k(t),s^j_k(t)\big)\in U^j_k$ with $\big(a^j_k(0),s^j_k(0)\big)=\big(a^j_k,0\big)$ and $(s^j_k)'(0)>0$ such that
    \begin{equation}\label{eq:local-form-bifbranch}
        f_k^{-1}(0)\cap U^j_k= (\Btriv\cap U^j_k)\cup\left\{\big(a^j_k(t),s^j_k(t)\big) : t\in (-\varepsilon,\varepsilon) \right\}.
    \end{equation}
        \item For all $a\neq a^j_k$, the trivial solution $(a,0)$ is locally unique for $f_k$, i.e., 
        if $a\neq a^j_k$, there exists an open neighborhood $V_k$ of $(a,0)$ in $(0,+\infty)\times (-1,1)$ such that $f_k^{-1}(0)\cap V_k=  \Btriv\cap V_k$.
    \end{enumerate}    
\end{proposition}

\begin{proof}
Differentiating \eqref{eq:f} in $s$ at $(a,0)$, we have
\begin{equation*}
\tfrac{\partial f_k}{\partial s}(a,0) = \tfrac{\partial}{\partial s}\dot\rho_{a,s}\big|_{s=0}(\ell_k(a,0))+\dot\rho_{a,0}(\ell_k(a,0))\tfrac{\partial\ell_k}{\partial s}(a,0)=\dot R_a(k\pi),
\end{equation*}
because $\ell_k(a,0)=k\pi$ by Remark~\ref{rem:taua0}, and $\dot\rho_{a,0}(\ell_k(a,0))=f_k(a,0)=0$. Recall that, for all $a>0$, the variational field $(R_a,T_a)=\frac{\dd}{\dd s}(\rho_{a,s},\theta_{a,s})\big|_{s=0}$ solves the ODE system \eqref{eq:jacobi-eqn-quotient} with initial conditions 
\begin{align*}
(R_a(0),T_a(0))&=(\beta'_a(0),0), \\
(\dot R_a(0),\dot T_a(0))&=\big(0,(\pi a)^{-2} \, \tfrac{\dd}{\dd s}c(\gamma_{a,s})\big|_{s=0}\big).
\end{align*}
Thus, we have that $R_a(t)=\beta'_a(0)\cos\left(\frac{2a}{\sqrt{a^2+1}} \,t\right)$, and hence $\dot R_a(k\pi)=0$ if and only if $\frac{2ak}{\sqrt{a^2+1}}\in\Z$. Since $0 < \frac{a}{\sqrt{a^2+1}}<1$ for all $a>0$, the only possible integer values for $\frac{2ak}{\sqrt{a^2+1}}$ are $j=1,\dots, 2k-1$, which correspond to $a=a^j_k$.
Altogether, $\tfrac{\partial f_k}{\partial s}(a,0) =0$ if and only if $a=a^j_k$ for some integer $0<j<2k$. Thus, (ii) follows from the Implicit Function Theorem, while (i) follows from \cite[Thm~2.2]{ellipsoids} since
\begin{equation*}
\tfrac{\partial^2 f_k}{\partial a\partial s}\big(a^j_k,0\big) =\tfrac{\partial}{\partial a} \dot R_a(k \pi) \Big|_{a=a^j_k} =\beta_{a^j_k}'(0)  \frac{(-1)^{j+1}(4k^2-j^2)^{3/2}\pi j}{4k^3}\neq 0.\qedhere
\end{equation*}
\end{proof}

\begin{remark}
The induced metric on $\Sigma_a\subset\Ss^3_a$ is flat, and isometric to $\R^2/\Gamma$ where $\Gamma$ is the integer lattice generated by $\left({\sqrt2}\pi a,0\right)$ and $\left(0,{\sqrt2}\pi \right)$. The Laplace spectrum of the $(k,l)$-fold covering of such a $2$-torus consists of the eigenvalues $2\left(\frac{j^2}{k^2a^2}+\frac{m^2}{l^2}\right)$, where $j,m\in\Z$. Moreover, we have $\Ric(\vec n)+\|A\|^2=\frac{8}{a^2+1}$, so the eigenvalues of the Jacobi operator of the $(k,l)$-covering of $\Sigma_a\subset\Ss^3_a$ are:
\begin{equation}\label{eq:eigenv-jacobi}
\phantom{, \qquad j,m\in\Z.}
 \textstyle   \lambda^{k,l}(j,m)=2\left(\frac{j^2}{k^2a^2}+\frac{m^2}{l^2}\right)-\frac{8}{a^2+1}, \qquad j,m\in\Z.
\end{equation}
The eigenfunctions with eigenvalue $\lambda^{k,l}(j,m)$ are $\{1\}\times\S^1$-invariant if and only if $m=0$, from which one recovers the degeneracy instants $a^j_k$ as the values of $a>0$ such that $\lambda^{k,l}(j,0)=0$. Moreover, 
$\frac{\partial^2 f_k}{\partial a\partial s}\big(a^j_k,0\big)=(-1)^j\,C\,\frac{\partial}{\partial a} \lambda^{k,l}(j,0)\big|_{a=a^j_k}$, for some $C>0$.

Analyzing \eqref{eq:eigenv-jacobi}, we see that if $\Sigma_a\subset\Ss^3_a$ is degenerate as an embedded minimal surface, i.e., $\lambda^{1,1}(j,m)=0$, then either $a = 1$, or the corresponding Jacobi field is $\S^1\times \{1\}$-invariant ($j=0$, $m=\pm1$, $a=\sqrt3$) or $\{1\}\times\S^1$-invariant ($j=\pm1$, $m=0$, $a=\frac{1}{\sqrt3}$). Meanwhile, if $k\geq 2$ or $l\geq 2$, then $\lambda^{k,l}(j,m)=0$ has non-$\S^1$-invariant solutions, i.e., with both $m\neq 0$ and $j\neq 0$, besides those with $a=1$. Thus, our $\S^1$-invariant setup detects all possible bifurcations of $\Sigma_a\subset\Ss^3_a$ through \emph{embedded} minimal tori, but there may be bifurcations through immersed minimal tori (that fail to be embedded) besides the $\S^1$-invariant ones studied here. In fact, since there is a jump in the Morse index of $\Sigma_a$ at such degeneracy instants $a=a_*$, the abstract bifurcation criterion in \cite{g-bif} implies that there are sequences $a_n\to a_*$ and $\Sigma_{(n)}\hookrightarrow \Ss^3_{a_n}$ of immersed minimal tori not congruent to $\Sigma_{a_n}$ that accumulate on $\Sigma_{a_*}$. However, without the $\S^1$-invariant setup, it is unclear how to extract global (in $a>0$) consequences from these local ``cluster-point'' bifurcations of $\Sigma_a$.
\end{remark}

\begin{remark}\label{rem:round}    
    The value $a=1$ is not a bifurcation instant for any $f_k$, since, by \eqref{eq:ajk}, $a^j_k=1$ if and only if $\frac{j}{k}=\sqrt 2$. Thus, given $k_0\in\N$, it follows from Proposition~\ref{prop:loc-bif}~(ii) that $\gamma_{a,0}$ and its first $k_0$ iterates are locally unique as closed geodesics in $\Omega_a$ for $a$ near $1$, since $\bigcap_{k\leq k_0} V_k$ is open. However, note that $\bigcap_{k\in\N} V_k$ need not be open. 
\end{remark}

\subsection{Bifurcation branches}
It follows from \eqref{eq:ajk} that the collection 
\begin{equation}\label{eq:allbif}
\mathfrak b=\bigcup_{k\in \N} \{a^j_k : 0 < j < 2k \}    
\end{equation}
of bifurcation instants forms a countable dense subset of $(0,+\infty)$, since it is the image of $(0,2)\cap \mathds Q$ under the diffeomorphism $(0,2)\ni q\mapsto q/\sqrt{4-q^2}\in(0,+\infty)$, and $a^j_k = a^{j'}_{k'}$ if and only if $\frac{j}{k}=\frac{j'}{k'}$. In order to uniquely identify each $a_*\in \mathfrak b$ as $a_*=a^j_k$ for some $j$ and $k$, we shall henceforth assume that $\gcd(j,k)=1$. In other words, we shall only consider $(a_*,0)$ as a bifurcation point for $f_k$ with the smallest possible $k$. By Proposition~\ref{prop:fk-zeros}~(iii), this means considering only bifurcations of the closed geodesic $\gamma_{a,0}$ through primitive closed geodesics, i.e., disregarding iterates.

Let $\mathcal B_{(j,k)}$ be the bifurcation branch for $f_k$ issuing from $(a^j_k,0)$, where $0<j<2k$ and $\gcd(j,k)=1$ per the convention above. In order to distinguish branches, we now determine a discrete-valued invariant for $f_k(a,s)=0$ , see \cite[Sec.~2.2.3]{ellipsoids}.

\begin{proposition}\label{prop:disjoint}
If $(a,s)\in\mathcal B_{(j,k)}\setminus\Btriv$, then $\gamma_{a,s}([-\ell_k(a,s),\ell_k(a,s)])$ is a closed geodesic in $\Omega_a$ that intersects the Clifford geodesic $\gamma_{a,0}$ at $2j$ instants, has winding number $k$ around~$O$, and self-intersects along $D(0)$ at $k-1$ instants. In particular, bifurcation branches are pairwise disjoint: $\mathcal B_{(j,k)}\cap \mathcal B_{(j',k')}=\emptyset$ if $(j,k)\neq(j',k')$.
\end{proposition}

\begin{proof}
If $(a,s)\in\mathcal B_{(j,k)}$, $s\neq 0$, is close to $(a^j_k,0)$, then $\rho_{a,s}\colon [0,\ell_k(a,s)]\to (0,L_a)$ is well-approximated to first order by $R_{a^j_k}\colon [0,k\pi]\to(0,L_a)$, $R_{a^j_k}(t)=\beta'_{a^j_k}(0)\cos\big(\frac{j}{k}t\big)$. In particular, $\rho_{a,s}(t)$ has exactly $j$ simple zeros on $[0,\ell_k(a,s)]$, i.e., $\gamma_{a,s}([0,\ell_k(a,s)])$ intersects $\gamma_{a,0}$ exactly $j$ times, so $\gamma_{a,s}([-\ell_k(a,s),\ell_k(a,s)])$ intersects it $2j$ times. Since intersections between distinct geodesics must be transverse, this remains true for all $(a,s)\in \mathcal B_{(j,k)}\setminus\Btriv$, proving the first assertion. The second assertion is a consequence of $\gamma_{a,s}\colon [-\ell_k(a,s),\ell_k(a,s)]\to\Omega_a\setminus\{O\}$ being regularly homotopic to $\gamma_{a,0}\colon [-k\pi,k\pi]\to\Omega_a\setminus\{O\}$, whose winding number around $O$ is equal to $k$, cf.~Proposition~\ref{prop:fk-zeros} (iii), and $\gamma_{a,s}(-\ell_{k'}(a,s))=\gamma_{a,s}(\ell_{k'}(a,s))$, $0<k'<k$, give rise to the $k-1$ self-intersections along $D(0)$. 
If $(j,k)\neq(j',k')$, then $\mathcal B_{(j,k)}$ and $\mathcal B_{(j',k')}$ cannot intersect at $(a,s)$ with $s\neq 0$ by the above, while intersection at $(a,0)$ would violate \eqref{eq:local-form-bifbranch} since $f^{-1}_{\max\{k,k'\}}(0)\cap U^j_k\cap U^{j'}_{k'}$ consists of a single curve besides $\Btriv$.
\end{proof}

The subsets $\mathcal B_{(j,k)}^\pm := \{(a,s)\in \mathcal B_{(j,k)} : \pm s>0\}$ are connected, by \eqref{eq:local-form-bifbranch} and $\mathcal B_{(j,k)} \cap\Btriv=\big\{\big(a^j_k,0\big)\big\}$.
For each $k\in\N$, there is an involutive homeomorphism
\begin{equation*}
\iota_k\colon f_k^{-1}(0)\to f_k^{-1}(0)
\end{equation*}
such that, for all $t\in [0,\ell_k(a,s)]$ and $(a,s)\in f_k^{-1}(0)$, 
\begin{equation*}
\gamma_{\iota_k(a,s)}(t)=\big(\rho_{a,s}(\ell_k(a,s)-t),k\pi-\theta_{a,s}(\ell_k(a,s)-t)\big)    
\end{equation*}
Clearly, $\gamma_{a,s}\big([0,\ell_k(a,s)]\big)$ and $\gamma_{\iota_k(a,s)}\big([0,\ell_k(a,s)]\big)$ are congruent geodesic segments, since they are mapped to one another by the reflection $\tau_{\frac\pi2}$. Note that $\ell_k\big(\iota_k(a,s)\big)=\ell_k(a,s)$ for all $(a,s)\in f_k^{-1}(0)$, and $\Pi(\iota_k(a,s))=a$ where $\Pi(a,s)=a$ is the projection onto the first factor. Moreover, $\iota_k$ fixes $\Btriv$ pointwise, i.e., $\iota_k(a,0)=(a,0)$ for all $a>0$. As $\mathcal B_{(j,k)}$ are connected and $\mathcal B_{(j,k)}\cap\Btriv=\big\{\big(a^j_k,0\big)\big\}$ by Proposition~\ref{prop:disjoint}, we have that $\iota_k$ maps $\mathcal B_{(j,k)}$ to itself. So, either $\iota_k(\mathcal B^\pm_{(j,k)})=\mathcal B^\mp_{(j,k)}$ or $\iota_k(\mathcal B^\pm_{(j,k)})=\mathcal B^\pm_{(j,k)}$.

\begin{proposition}\label{prop:pm}
    Let $k\in\N$ and $0<j<2k$ with $\gcd(j,k)=1$. If $j$ is even, then $\iota_k(\mathcal B^\pm_{(j,k)})=\mathcal B^\pm_{(j,k)}$; if $j$ is odd, then $\iota_k(\mathcal B^\pm_{(j,k)})=\mathcal B^\mp_{(j,k)}$.
\end{proposition}

\begin{proof}
By the proof of Proposition~\ref{prop:disjoint}, if $(a,s)\in \mathcal B_{(j,k)}$, then $\gamma_{a,s}\big([0,\ell_k(a,s)]\big)$ intersects $\gamma_{a,0}$ exactly $j$ times, so the endpoints of $\gamma_{a,s}([0,\ell_k(a,s)])$ lie in the same component of $\Omega_a\setminus\gamma_{a,0}$ if and only if $j$ is even. In other words, if $(a,s)\in \mathcal B_{(j,k)}$, then $\rho_{a,s}(0)-\rho_0>0$ if and only if $(-1)^j(\rho_{a,s}(\ell_k(a,s))-\rho_0)>0$ hence if and only if $(-1)^j(\rho_{\iota_k(a,s)}(0)-\rho_0)>0$. On the other hand, if $(a,s)\in \mathcal B^\pm_{(j,k)}$ is close to $\big(a^j_k,0\big)$, then $\pm (\rho_{a,s}(0)-\rho_0)=\pm (\beta_a(s)-\beta_a(0))>0$, because $\beta'_a(0)>0$. 
\end{proof}

\subsection{Global behavior}
The only points $(a,s)\in \mathcal B_{(j,k)}$ near the boundary of the strip $(0,+\infty)\times(-1,1)$ must have $a\approx 0$ or $a\approx+\infty$ as a consequence of the next:

\begin{proposition}\label{prop:no-sol-near-bdy}
If a convergent sequence $(a_n,s_n)\in\mathcal B_{(j,k)}$ has $|s_n|\nearrow1$, then either $a_n\searrow0$ or $a_n\nearrow+\infty$. In particular, the restriction to $\mathcal B_{(j,k)}$ of the projection $\Pi\colon (0,+\infty)\times (-1,1)\to(0,+\infty)$, $\Pi(a,s)=a$, is a proper map.
\end{proposition}

\begin{proof}
The only geodesics in $\Omega_a$ that converge to the boundary $\partial\Omega_a$ or that pass through the central point $O$ are the diameters \eqref{eq:diam}. Thus, if a sequence of points $(a_n,s_n)\in\mathcal B_{(j,k)}$ has $a_n\to a_\infty\in (0,+\infty)$ and $s_n\nearrow 1$, respectively $s_n\searrow-1$, then the corresponding closed geodesics $\gamma_{a_n,s_n}([-\ell_k(a_n,s_n),\ell_k(a_n,s_n)])$ converge graphically to the diameter $D(0)$, respectively $D(\frac\pi2)$, with multiplicity $2k$ in $\Omega_{a_\infty}$. So, for large $n$, the Clifford geodesic $\gamma_{a_n,0}$ intersects these closed geodesics $4k$ times, but this contradicts Proposition~\ref{prop:disjoint} because $2j<4k$. 
\end{proof}

\begin{remark}
    The case $j=k=1$ of Proposition~\ref{prop:no-sol-near-bdy} can be alternatively proved using the compactness theorem of Choi--Schoen~\cite{choi-schoen}. 
    Indeed, a convergent sequence $(a_n,s_n)\in\mathcal B_{(1,1)}$ with $a_n\to a_\infty\in (0,+\infty)$ and $|s_n|\nearrow1$ would correspond to a sequence of embedded minimal tori in $\Ss^3_{a_n}$ that converges to a planar minimal sphere in $\Ss^3_{a_\infty}$ with multiplicity $2$, which is impossible by~\cite{choi-schoen}.
\end{remark}

We are now ready to prove the Theorem and Corollary from the Introduction.

\begin{proof}[Proof of Theorem]
By Proposition~\ref{prop:loc-bif}~(i) and \eqref{eq:allbif}, the collection $\mathfrak b$ of bifurcation instants for $\Sigma_a$ forms a dense subset of $(0,+\infty)$. We then apply the Rabinowitz global bifurcation theorem~\cite{rabinowitz}, see \cite[Thm~2.5]{ellipsoids} for a formulation adapted to the present setup, to each bifurcation instant $a_*=a^j_k$. Hypotheses (1) and (2) of \cite[Thm~2.5]{ellipsoids} are satisfied as a consequence of Proposition~\ref{prop:no-sol-near-bdy}, so the curve $(-\varepsilon,\varepsilon)\ni t\mapsto \big(a^j_k(t),s^j_k(t)\big)\in (0,+\infty)\times (-1,1)$ in \eqref{eq:local-form-bifbranch} can be extended to a piecewise real analytic curve $\R\ni t\mapsto \big(a^j_k(t),s^j_k(t)\big)\in (0,+\infty)\times (-1,1)$ with $s^j_k(t)=0$ if and only if $t=0$, whose restriction to $(0,+\infty)$ takes values in $\mathcal B^+_{(j,k)}$. Moreover, \cite[Thm~2.5]{ellipsoids} yields a dichotomy for the global behavior of branches. Alternative (I) is reattachment to the trivial branch; by Proposition~\ref{prop:loc-bif}, it would have to occur at some other bifurcation instant $a_{**}=a^{j'}_{k'}$, with $(j,k)\neq (j',k')$, in contradiction to Proposition~\ref{prop:disjoint}. Thus, alternative (II) holds, i.e., $\mathcal B_{(j,k)}$ are noncompact. In view of Proposition~\ref{prop:no-sol-near-bdy}, this implies either $\lim_{t\to+\infty} a^j_k(t)=0$ or $\lim_{t\to+\infty} a^j_k(t)=+\infty$. 
The claim that minimal tori bifurcating at different values of $a\in\mathfrak b$ are pairwise noncongruent follows from Proposition~\ref{prop:disjoint}, and embeddedness of those bifurcating at $a^1_1=\frac{1}{\sqrt3}$ follows from Proposition~\ref{prop:fk-zeros} (ii).
\end{proof}

\begin{proof}[Proof of Corollary]
The $\{1\}\times\S^1$-invariant immersed minimal tori in $\Ss^3_a$ given by the preimage of the geodesic $\gamma_{a,s}$ with $(a,s)\in\mathcal B_{(j,k)}\setminus \{(a^j_k,0)\}$ are not congruent to the Clifford torus $\Sigma_a$ by Proposition~\ref{prop:fk-zeros}. If $j=k=1$, these tori are embedded and hence $\lim_{t\to+\infty} a^1_1(t)=0$, for otherwise $\lim_{t\to+\infty} a^1_1(t)=+\infty$ so $a^1_1(t_*)=1$ for some $t_*>0$ in contradiction to the fact that the unique embedded minimal torus in the round sphere $\Ss^3_1$ is the Clifford torus~\cite{brendle}. The branches $\mathcal B_{(j,k)}$ are pairwise disjoint by Proposition~\ref{prop:disjoint}, so $\lim_{t\to+\infty} a^j_k(t)=0$ for all $0<j<k$, which implies that every ellipsoid $\Ss^3_a$ with $a<a^1_1=\frac{1}{\sqrt3}$ contains infinitely many pairwise noncongruent $\{1\}\times\S^1$-invariant immersed minimal tori, including the embedded minimal torus given by the preimage of $\gamma_{a^1_1(t),s^1_1(t)}$. For $k<j<2k$, either $\lim_{t\to+\infty} a^j_k(t)=0$ for all such $j,k$, hence there are infinitely many pairwise noncongruent $\{1\}\times\S^1$-invariant immersed minimal tori in $\Ss^3_a$ for all $a>0$, or there exists a smallest $j_+/k_+\in (1,2)\cap\mathds Q$ such that $\lim_{t\to+\infty} a^{j_+}_{k_+}(t)=+\infty$, see Figure~\ref{fig:branches}. In that case, Proposition~\ref{prop:disjoint} implies $\lim_{t\to+\infty} a^{j}_{k}(t)=+\infty$ for all $j,k$ with $a^j_k>a^{j_+}_{k_+}>\frac{1}{\sqrt3}$, so the same conclusion holds except possibly at $a=a^{j_+}_{k_+}$.
\end{proof}

\begin{figure}[!ht]
\begin{tikzpicture}[scale=0.85]
    \draw [->,thick,-latex] (0,-2.5) -- (0,2.5) node [right] {$s$};
    \draw [->,thick,-latex] (0,0) -- (11.35,0) node [right] {$a$};
    
    \draw (0,2) node [left] {$1$};
    \draw (0,0) node [left] {$0$};
    \draw (0,-2) node [left] {$-1$};
    \draw [-,thick,dashed] (0,-2) -- (11,-2);
    \draw [-,thick,dashed] (0,2) -- (11,2);

    \draw (10.5,0.3) node {$\color{red}\mathcal B_\triv$};
    \draw [-,thick,red] (0,0) -- (11,0);

    \draw (9.5,1.5) node {$\{(a^j_k(t),s^j_k(t))\}\subset\mathcal B_{(j,k)}$};
    \draw (0.5,0) node [above] {$\cdots$};
    \foreach \x in {10,...,90}
    {\draw [-] (0.1+0.1*\x,0) to[out=90,in=-0.2*\x,looseness=0.5] (0,1.99);
    \draw [-] (0.1+0.1*\x,0) to[out=-90,in=0.2*\x,looseness=0.5] (0,-1.99);}

    \draw [-,blue,thick] (0.1+0.1*40,0) to[out=90,in=-0.2*40,looseness=0.5] (0,1.99);
    \draw [-,blue,thick] (0.1+0.1*40,0) to[out=-90,in=0.2*40,looseness=0.5] (0,-1.99);

    \draw (9.11,0) node {$_|$};
    \draw (9.3,0) node [below] {$a^j_k$};
    \draw (9.5,0) node [above] {$\cdots$};
\end{tikzpicture}

\begin{tikzpicture}[scale=0.85]
    \draw [->,thick,-latex] (0,-2.5) -- (0,2.5) node [right] {$s$};
    \draw [->,thick,-latex] (0,0) -- (11.35,0) node [right] {$a$};
    
    \draw (0,2) node [left] {$1$};
    \draw (0,0) node [left] {$0$};
    \draw (0,-2) node [left] {$-1$};
    \draw [-,thick,dashed] (0,-2) -- (11,-2);
    \draw [-,thick,dashed] (0,2) -- (11,2);

    \draw (11,0.3) node {$\color{red}\Btriv$};
    \draw [-,thick,red] (0,0) -- (11,0);

    \draw (0.5,0) node [below] {$\cdots$};
    \foreach \x in {10,...,50}
    {\draw [-] (0.1+0.1*\x,0) to[out=90,in=-0.2*\x,looseness=0.7] (0,1.99);
    \draw [-] (0.1+0.1*\x,0) to[out=-90,in=0.2*\x,looseness=0.7] (0,-1.99);}

    \draw [-,blue,thick] (0.1+0.1*40,0) to[out=90,in=-0.2*40,looseness=0.7] (0,1.99);
    \draw [-,blue,thick] (0.1+0.1*40,0) to[out=-90,in=0.2*40,looseness=0.7] (0,-1.99);

    \foreach \x in {0,...,48}
    {\draw [-] (5.15+0.1*\x,0) to[out=90,in=186+0.3*\x,looseness=0.7] (11,1.75-0.02*\x);
    \draw [-] (5.15+0.1*\x,0) to[out=-90,in=174-0.3*\x,looseness=0.7] (11,-1.75+0.02*\x);}
    \draw (10.75,0) node [below] {$\cdots$};
    \draw (5.14,0) node {$_|$};
    \draw (5.14,-0.4) node [below] {$a^{j_+}_{k_+}$};
\end{tikzpicture}
\caption{Schematic illustration depicting qualitative behavior of bifurcation branches $\mathcal B_{(j,k)}$ for small $k$, including $\mathcal B_{(1,1)}$ in blue, according to whether $\lim_{t\to+\infty} a^j_k(t)=0$ for all $j,k$ (top) or not (bottom). The expected (conjectural) behavior is the top illustration. 
These images are not the result of any numerical experiment.}\label{fig:branches}
\end{figure}

\begin{remark}
If $a<\frac{1}{\sqrt3}$, then the vertical line $\{a\}\times (-1,1)$ intersects $\mathcal B_{(1,1)}$ in at least two points $(a,s_\pm)$, where $s_-<0<s_+$, but we cannot exclude the possibility that the corresponding embedded minimal tori in $\Ss^3_a$ are congruent. In fact, if $(a,s_-)=\iota_1(a,s_+)$, then $\gamma_{a,s_\pm}$ are congruent via the reflection $\tau_{\frac\pi2}$, see the discussion preceding Proposition~\ref{prop:pm}.
\end{remark}

\begin{remark}
While $a=1$ is a barrier for the bifurcation branch $\mathcal B_{(1,1)}$ of embedded minimal tori in $\Ss^3_a$, it is not a barrier for $\mathcal B_{(j,k)}$, $(j,k)\neq (1,1)$. In fact, infinitely many such branches cross $a=1$ as a consequence of Remark~\ref{rem:round}. By a result of Brendle~\cite{brendle-mrl}, all Alexandrov immersed minimal tori in the round sphere $\Ss^3_1$ are rotationally invariant. There are infinitely many such tori, known as \emph{\^Otsuki tori}, first described in \cite{otsuki}, see also \cite[Thm.~8]{hsiang-lawson} or \cite[Thm.~1.4]{brendle-survey}, and all must be congruent to tori in some $\mathcal B_{(j,k)}$ by Remark~\ref{rem:allinbranches}.
\end{remark}

\begin{remark}
    For odd $k\in\N$, the minimal tori in $\mathcal B_{(j,k)}$ are invariant under the isometry of $\Ss^3_a$ given by $(z,w)\mapsto (\bar z,w)$, so they yield free boundary immersed minimal annuli in the ellipsoidal hemispheres $(\Ss^3_a)^\pm =\{(z,w)\in \Ss^3_a : \pm\operatorname{Im} z\geq 0\}$. Similar free boundary minimal annuli have been recently constructed inside geodesic balls of $\Ss^3_1$ in \cite{cerezo-fernandez-mira} and inside ellipsoidal domains in $\R^3$ in \cite{schulz}.
\end{remark}

\providecommand{\bysame}{\leavevmode\hbox to3em{\hrulefill}\thinspace}
\providecommand{\MR}{\relax\ifhmode\unskip\space\fi MR }
\providecommand{\MRhref}[2]{%
  \href{http://www.ams.org/mathscinet-getitem?mr=#1}{#2}
}
\providecommand{\href}[2]{#2}

\end{document}